\documentclass{llncs}     % onecolumn (ditto)
\usepackage{graphicx}
\usepackage{cite}
\usepackage{amsfonts}
\usepackage{amsmath}
\usepackage[all]{xy}
\usepackage{amssymb}
\usepackage{color}
\usepackage{enumerate}
\usepackage{eucal}

\usepackage{amsthm}
\usepackage{thmtools}
\usepackage[utf8]{inputenc}
\usepackage[T1]{fontenc}

\allowdisplaybreaks[3]

\declaretheorem[bodyfont = \normalfont,name=Notation]{Mnotation}
\declaretheorem[style=mystyle,name=Example]{Mexample}
\def\justl#1#2{\\
        &#1& \rule{2em}{0pt}  \{
        \mbox{\rule[-.7em]{0pt}{1.8em} \footnotesize{#2} \} } \\ && }
\def\redu#1#2{{#1} \hspace{-0.3em} \upharpoonright_{#2}}
\def\comp{\mathbin{\boldsymbol{\cdot}}}
\def\pv#1#2{\langle #1 \rangle #2}
\newcommand{\ie}{\emph{i.e.}}

\DeclareMathOperator{\suc}{suc}

\begin{document}

\title{Asymmetric Combination of Logics is Functorial: A Survey}

\titlerunning{Asymmetric Combination Of Logics Is
  Functorial} % if too long for running head

\author{Renato Neves  \inst{1}    \and
        Alexandre Madeira \inst{1} \and
        Luis S. Barbosa \inst{1}  \and
        Manuel A. Martins \inst{2}
}


\institute{
  INESC TEC (HASLab) \& Universidade do Minho, Portugal \\
  \email{nevrenato@di.uminho.pt, amadeira@inesctec.pt,
    lsb@di.uminho.pt} \and
  CIDMA -- Dep. of Mathematics, Universidade de Aveiro, Portugal \\
  \email{martins@ua.pt} }

\maketitle

\begin{abstract}
  % The rigorous development of complex, heterogenous software systems
  % requires the combined use of different logics to capture different
  % sorts of features and requirements.
  Asymmetric combination of logics is a formal process that develops
  the characteristic features of a specific logic on top of another
  one. Typical examples include the development of temporal, hybrid,
  and probabilistic dimensions over a given base logic. These examples 
  are surveyed in the paper under a particular perspective --- that 
  this sort of combination of logics  possesses a functorial
  nature. Such a view gives rise to several interesting questions.
  They range from the problem of combining
  translations (between logics), to that of ensuring property
  preservation along the process, and the way different asymmetric
  combinations can be related through appropriate natural
  transformations.

% This paper seeks for an abstract formulation of this process in an
% institutional setting, paving the way to the systematic study of the
% resulting, combined logic, namely in what concerns property
% preservation.
\end{abstract}

\section{Introduction}
\label{intro}

\subsection{Motivation and Context}

It is well known that software's inherent high complexity renders
formal design and analysis a difficult challenge, still largely unmet
by the current engineering practices.  Often, in fact, the formal
specification of a non trivial software system calls for multiple
logics so that specific types of requirements and design issues can be
captured: if properties of data structures are typically encoded in
an equational framework, behavioural issues will call for some sort of
modal or temporal logic, whereas probabilistic reasoning will be
required in order to predict or analyse faulty behaviour in
distributed systems.

This fact explains the growing interest in the systematic combination
of logics, an area whose overall aim can be summed up in a simple
methodological principle: \emph{identify the different natures of the
  requirements to be formalised, and combine whatever logics are
  suitable to handle them into a single logic for the whole
  system}. Its potential was already stressed in the eighties by
J. Goguen and J. Meseguer, and the whole programme started to gain
prominence in the following decade (\emph{cf.}
\cite{blackburn97,equality_logicalprogramming}).

The current paper surveys a specific type of combination of logics, 
called \emph{asymmetric}, in which
the characteristic features of a logic are developed on top of
another one.
% (whenever found suitable we will drop the qualifier
% \emph{asymmetric} or the subject \emph{logics} in the expression).
Probably the most famous example is the process of
\emph{temporalisation} \cite{finger}, in which the features of a
temporal logic are added to another logic; the latter is often referred
to as the \emph{base logic} in order to distinguish the original
machinery from the one added along the process. In brief,
temporalisation adds a temporal dimension to the models of a given
logic and syntactical machinery to suitably handle this added
dimension. The \emph{hybridisation} \cite{hybridisation} and
\emph{probabilisation} \cite{probabilization} processes are more
recent examples. The former develops a \emph{hybrid} logic
\cite{hybridlogics} on top of the base one whereas the latter adds
probabilistic features. Other examples include \emph{quantisation}
\cite{caleiro06} and \emph{modalisation} \cite{modal_diaconescu},
bringing into the picture features of quantum and modal logic,
respectively.

Is there a common characterisation of these different combinations,
able to provide a suitable setting to discuss their properties at a
generic level?  Such is the question addressed in this paper through
the identification of their common \emph{functorial nature}. This
perspective structures the whole survey presented here.

Our approach is based on the theory of \emph{institutions}
\cite{institutions}, an abstract characterisation of logical systems
that encompasses syntax, semantics, and satisfaction.  Put forward by
J. Goguen and R. Burstall in the late seventies, its original aim was
to develop as much Computing Science as possible in a general, uniform
way, independently of any particular logical system, in response to
the \emph{``population explosion among the logical systems used in
  Computing Science''} \cite{institutions}. Since then this goal has
been achieved to an extent even greater than originally thought.
Indeed, institutions underlie the foundations of algebraic
specification methods, and are most useful in handling and combining
different sorts of logical systems.  The universal character and
resilience of institutions is witnessed by the wide set of logics
formalised and subsequently explored within the framework.  Examples
go from standard classical logics, to more unconventional ones,
typically capturing modern specification and programming paradigms ---
examples include \emph{process algebras} \cite{fiadeiro07},
\emph{temporal logics} \cite{cengarle98}, the \textsc{Alloy} language
\cite{neves14}, coalgebraic logics \cite{modalinst}, functional and
imperative languages \cite{foundations_algspec}, among many others.

\subsection{Contributions and Roadmap}

Institutions are objects of a well known category $\mathbf{I}$ whose
arrows are the so-called \emph{institution comorphisms} (\emph{cf.}
\cite{mossa07, foundations_algspec}).  In this
setting we argue that an asymmetric combination of logics can, very
often, be seen as an \emph{endofunctor} over $\mathbf{I}$. Three
examples (\emph{temporalisation}, \emph{hybridisation}, and
\emph{probabilisation}) are discussed in detail, with their
definitions (slightly) reworked to fit in the general picture. Such a
functorial perspective has several advantages: an interesting one is
the possibility to lift the combination process from logics to their
translations, which allows for the characterisation of natural
transformations between asymmetric combinations. Another interesting
possibility is the study of adjoints, and preservation of properties
such as conservativity, equivalence, and (co)limits.

We initiate this survey with a brief overview of common approaches to
combination of logics, in Section \ref{sc:comb}. From there on, the
focus is placed on asymmetric combinations and the characterisation of
their functorial nature.

Thus, in Section \ref{sc:inst} we recall the category of institutions
$\mathbf{I}$ and revisit the three combinations of logics discussed in
the paper.  Then, in Section \ref{sc:fun}, these examples are made
functorial. For the sake of simplicity and conciseness, we define an
institutional notion of asymmetric combination and make, to a large
extent, the necessary proofs at this level of abstraction. We stress,
however, that the paper's main objective is not to introduce such a
notion, but rather to survey the functorial nature of a number of
asymmetric combinations and to show that the functorial perspective
paves the way to several interesting mechanisms and research lines.

In the same section we study property preservation by these three
(new) functors in what concerns conservativity (an important property
in the validation of specifications) and the equivalence of
institutions. We also discuss natural transformations between
asymmetric combinations.  Finally, in Section \ref{sc:con}, we conclude and
suggest future lines of research.

This paper assumes a basic knowledge of Category Theory. Whenever
found suitable, we will omit subscripts in natural transformations and
denote the underlying class of objects of a category $\mathbf{C}$ by
$|\mathbf{C}|$ or just $\mathbf{C}$.

\section{Combination of logics: A brief overview}\label{sc:comb}

The entry on \emph{Combining Logics} in the \emph{Stanford
  Encyclopedia of Philosophy} \cite{sep-logic-combining} stresses the
role of Computing Science applications as a main driving force for
research in obtaining new logical systems from old, integrating
features and preserving properties to a reasonable extent: \emph{``One
  of the main areas interested in the methods for combining logics is
  software specification. Certain techniques for combining logics were
  developed almost exclusively with the aim of applying them to this
  area.''}  The aforementioned hybridisation and temporalisation
methods, for example, were originally developed with concrete
applications to Computing Science in mind, but interestingly they can
be more broadly understood as a specific way of combining logics at a
model theoretical
level. % Actually, they classify as \emph{tools for simplifying problems
  % involving heterogeneous reasoning} \cite{sep-logic-combining}.

As already mentioned, an asymmetric combination of logics
develops specific features of a logic `on top'  of another one. % This 
% follows exactly the same steps, and to a certain extent extends,
% previous work by R. Diaconescu and P. Stefaneas
% \cite{modal_diaconescu} on `modalisation' of institutions, which
% endows systematically institutions with Kripke semantics for
% standard modalities.  R. Fajardo and M. Finger introduced in
% \cite{DBLP:conf/aiml/FajardoF02} an alternative method to modalise
% logics, and proved preservation of both completeness and
% decidability of the source logics.  Other examples, in a similar
% research line, include the `temporalisation' of logics introduced by
% M. Finger and D. Gabbay in \cite{finger} and the more recent
% `probabilisation' of logics introduced by P. Baltazar in
% \cite{probabilization}.  The work of A. Costa Leite on what he calls
% \emph{paraconsistentization of logics} \cite{CLPhD} goes in a
% similar direction investigating how the paraconsistent counterpart
% of an arbitrary logic can be obtained
This sort of combination was generalised by C. Caleiro, A. Sernadas
and C. Sernadas in \cite{parametrisation}, in a method called
\emph{parameterisation}.  In brief, a logic is parametrised by another
one if the atomic part of the former is replaced by the latter: thus, the
method distinguishes a parameter to fill (the atomic part), a
parametrised logic (the `top' logic) and a parameter logic (the logic
inserted within). More recently, J. Rasga \emph{et al.}
\cite{importinglogics} proposed a method for importing logics
by exploiting a graph-theoretic approach.

From a wider perspective, combination of logics is increasingly
recognised as a relevant research domain, driven not only by
philosophical enquiry on the nature of logics or strict mathematical
questions, but also from applications in Computing Science and
Artificial Intelligence. The first methods appeared in the context of
modal logics.  This includes \emph{fusion} of the underlying languages
\cite{Thomason84}, pioneered by M. Fitting in a 1969 paper combining
alethic and deontic modalities \cite{fitting69}, and \emph{product of
  logics} \cite{productslogics}.  Both approaches can be characterised
as \emph{symmetric}. Product of logics, for example, amounts to
pairing the Kripke semantics, \ie\ the accessibility relations, of
both logics.  With a wider scope of application, \ie\ beyond modal
logics, \emph{fibring} \cite{fibringG} was originally proposed by
D. Gabbay, and contains fusion as a particular case. From a syntactic
point of view the language of the resulting logic is freely generated
from the signatures of the combined logics, symbols from both of them
appearing intertwined in an arbitrary way.

Reference \cite{caleirofibring05} offers an excellent roadmap for the  
several variants of fibring in the literature. A particularly relevant 
evolution was the work of A. Sernadas and his collaborators resorting 
to universal constructions from category theory to characterise 
different patterns of connective sharing, as documented in \cite{SSC99}. 
In the simplest case, where no constraint is imposed by sharing, fibring is 
the least extension of both logics over the coproduct of their signatures, 
which basically amounts to a coproduct of logics. This approach, usually 
referred to as \emph{algebraic fibring}, makes heavy use of categorial 
constructions as a source of genericity to provide more general and 
wide applicable methods. 

%The use of the theory of institutions \cite{bg80} as a foundation for 
%characterising this sort of combinations of logics has a similar motivation: 
%going categorial is going generic. Actually, a proper setting to discuss the 
%generation of new logics form old, and to identify the sort of properties 
%preserved or reflected along such a process, always requires a generic 
%definition of what a logic and a logic system is. The notion of a universal logic 
%\cite{BezUL} provides an alternative to institutions in this role 
%(see, e.g. \cite{CLcombination} for a discussion on logic combination 
%from this perspective), although not so well known outside the frontiers of pure logic research. 
%

\section{Asymmetric combination of logics
  (institutionally)}\label{sc:inst}
\subsection{Institutions} 
Let us recall the core notions of the theory of institutions and
revisit the three working examples of combinations.

\begin{definition} 
    An institution $\mathcal{I}$ is a tuple $( Sign^{\mathcal{I}}$,
    $Sen^{\mathcal{I}}$, $Mod^{\mathcal{I}}$,
    $(\models_{\Sigma}^{\mathcal{I}})_{\Sigma \in |Sign^{\mathcal{I}}|})$ where
			
    \begin{itemize}
    \item
          $Sign^{\mathcal{I}}$ is a category whose objects are
          signatures and arrows signature morphisms.

    \item 
         $Sen^{\mathcal{I}}$ : $Sign^{\mathcal{I}}$ $\rightarrow$
         {$\mathbf{Set}$}, is a functor that for each signature $\Sigma \in
         |Sign^\mathcal{I}|$ returns a set of $\Sigma$-sentences,

    \item 

        $Mod^{\mathcal{I}}$ : $(Sign^{\mathcal{I}})^{op}$ $\rightarrow$
        {$\mathbf{Cat}$}, is a functor that for each signature $\Sigma \in
        |Sign^\mathcal{I}|$ returns a category whose objects are
        $\Sigma$-models and the arrows are $\Sigma$-model homomorphisms.

    \item 

        $\models_{\Sigma}^{\mathcal{I}}$ $\subseteq$
        $|Mod^{\mathcal{I}}(\Sigma)| \times Sen^{\mathcal{I}}(\Sigma)$, is a
        satisfaction relation such that for each signature morphism $\varphi$
        : $\Sigma \rightarrow \Sigma'$ the following property holds

        \begin{center}

       $Mod^{\mathcal{I}}(\varphi)(M)
       \models_{\Sigma}^{\mathcal{I}} \rho$ iff $M
       \models_{\Sigma'}^{\mathcal{I}} Sen^{\mathcal{I}}(\varphi)(\rho)$ 

        \end{center}

        \noindent
	for any $M \in |Mod^{\mathcal{I}}(\Sigma')|$, $\rho \in
        Sen^{\mathcal{I}}(\Sigma)$.  Diagrammatically,
        \[ \xymatrix{ \Sigma \ar[d]_{\varphi} & Mod^{\mathcal{I}}(\Sigma)
        \ar@{-}[rr]^{\models^{\mathcal{I}}_\Sigma} & &
        Sen^\mathcal{I}(\Sigma)\ar[d]^{Sen^\mathcal{I}(\varphi)}
        \\ \Sigma' &
        Mod^{\mathcal{I}}(\Sigma')\ar[u]^{Mod^\mathcal{I}(\varphi)}
        \ar@{-}[rr]_{\models^\mathcal{I}_{\Sigma'}}
        &&Sen^{\mathcal{I}}(\Sigma')\\ }
        \]

      \end{itemize}
      If the tuple does not necessarily respects the satisfaction
      condition above then we call it a \emph{pre-institution}.
\end{definition}

\begin{Mnotation}	
  In the sequel we will refer to $Mod^{\mathcal{I}}(\varphi)(M)$ as the
  $\varphi$-reduct of $M$ and denote it by $\redu{M}{\varphi}$.  When
  clear from the context, both the subscript and superscript in the
  satisfaction relation will be dropped.
\end{Mnotation}

% \begin{Mexample} Propositional Logic $(PL)$

% \begin{itemize}
%     \item 
%     \textsc{Signatures}. $Sign^{PL}$ is a category whose
%     objects are sets of propositional symbols $P$, and arrows are
%     functions.

%    \item 
%    \textsc{Sentences}. For each signature $P \in |Sign^{PL}|$,
%     $Sen^{PL} (P)$ is the set of expressions built from symbols
%     in $P$, and closed under the typical Boolean connectives.

%    \item \textsc{Models}. For each signature
%    $P \in |Sign^{{PL}}|, Mod^{{PL}} (P)$ is the discrete
%    category whose objects are subsets of propositional
%    symbols in $P$.

%    \item 
%    \textsc{Satisfaction}. Satisfaction of sentences by models is the
%    usual Boolean satisfaction.

%    \end{itemize}

% \end{Mexample}

    \begin{definition} 

      Consider two institutions $\mathcal{I},\mathcal{I}'$.  A
      \emph{comorphism} $(\Phi,\alpha,\beta): \mathcal{I}
      \rightarrow \mathcal{I}'$ is a triple such that

    \begin{itemize}

       \item 

       $\Phi$: $Sign^\mathcal{I} \rightarrow Sign^{\mathcal{I}'}$ is a
       functor,

       \item 

       $\alpha$: $Sen^\mathcal{I} \rightarrow Sen^{\mathcal{I}'} \comp
       \Phi$ is a natural transformation, 

       \item 

         $\beta$: $Mod^{\mathcal{I}'} \comp \Phi^{op} \rightarrow Mod^\mathcal{I}$
         is a natural transformation \footnote{$(\_)^{op}$ applied to
           a functor $F : \mathbf{C} \rightarrow \mathbf{D}$
           induces a functor
           $F^{op} : \mathbf{C}^{op} \rightarrow \mathbf{D}^{op}$
           such that for any object or arrow $a$ in $\mathbf{C}$,
           $F^{op}(a) = F(a)$.},

       \item

       and for any $\Sigma \in |Sign^\mathcal{I}|$, $M \in |Mod^{\mathcal{I}'}
      \comp \Phi^{op} \> (\Sigma) |$ and $\rho \in Sen^\mathcal{I}(\Sigma)$
    		\begin{center} 

                $\beta_{\Sigma}(M) \models^\mathcal{I}_{\Sigma} \rho$ iff $M
                \models^{\mathcal{I}'}_{\Phi(\Sigma)} \alpha_{\Sigma}(\rho)$

              \end{center}
      \noindent
      Diagrammatically, for each $\Sigma \in |Sign^{\mathcal{I}}|$
      \[ \xymatrix{ Mod^{\mathcal{I}}(\Sigma)
          \ar@{-}[rr]^{\models^{\mathcal{I}}_\Sigma} &
          &Sen^{\mathcal{I}}(\Sigma)\ar[d]^{\alpha_\Sigma} \\
          Mod^{\mathcal{I}'} \comp \Phi^{op}
          (\Sigma)\ar[u]^{\beta_\Sigma}\ar@{-}[rr]_{\models^{\mathcal{I}'}_{\Phi(\Sigma)}}
          &&Sen^{\mathcal{I}'} \comp \Phi (\Sigma)\\ }
      \]	
       \end{itemize}
\end{definition}

\begin{definition}

  Let us consider two comorphisms $(\Phi_1,\alpha_1,\beta_1) :
  \mathcal{I} \rightarrow \mathcal{I}'$, and
  $(\Phi_2,\alpha_2,\beta_2) : \mathcal{I}' \rightarrow
  \mathcal{I}''$. Their composition $(\Phi_2,\alpha_2,\beta_2) \> ; \>
  (\Phi_1,\alpha_1,\beta_1) : \mathcal{I} \rightarrow
  \mathcal{I}''$ is defined as $(\Phi_2,\alpha_2,\beta_2) \> ; \>
  (\Phi_1,\alpha_1,\beta_1) \triangleq (\Phi_2 \comp \Phi_1, \>
  (\alpha_2 \circ 1_{\Phi_1}) \comp \alpha_1, \> \beta_1 \comp
  (\beta_2 \circ 1_{\Phi_1^{op}}) )$
  where the white circle denotes the \emph{Godement (horizontal)}
  composition of natural transformations.
  Thus,
   \begin{flalign*}
       & \Phi_2 \comp \Phi_1 \> : \> Sign^\mathcal{I} \rightarrow
       Sign^{\mathcal{I}''}, \\ & (\alpha_2 \circ 1_{\Phi_1}) \comp \alpha_1 \>
       : \> Sen^\mathcal{I} \rightarrow Sen^{\mathcal{I}''} \comp \Phi_2 \comp
       \Phi_1, \\ & \beta_1 \comp (\beta_2 \circ 1_{\Phi_1^{op}}) \> : \> 
       Mod^{\mathcal{I}''} \comp \Phi_2^{op} \comp \Phi_1^{op} \rightarrow
       Mod^{\mathcal{I}}.
  \end{flalign*}

\end{definition}

\noindent
Each institution $\mathcal{I}$ has as the identity comorphism the
triple $(1_{Sign^\mathcal{I}}, 1_{Sen^\mathcal{I}},
1_{Mod^\mathcal{I}})$.

As mentioned in the Introduction, institutions and respective
comorphisms form a category $\mathbf{I}$.

\subsection{An institutional rendering of asymmetric combinations of
  logics}

Consider the following abstract characterisation of what is an
asymmetric combination of logics.  Start with arbitrary categories
$Sign_1$, $Sign_2$, and two functors
\begin{flalign*}
   M^{\mathcal{C}} : (Sign_1)^{op} \rightarrow \mathbf{Cat}, \> \>
   M^{\mathcal{I}} : (Sign_2)^{op} \rightarrow \mathbf{Cat}.
\end{flalign*}
Assume that, for each $\Delta \in |Sign_1|$, there is a functor
$U_{(M^{\mathcal{C}},\Delta)} : M^{\mathcal{C}} (\Delta) \rightarrow
\mathbf{Set}$.  Whenever no ambiguities arise, we will drop the
subscript of $U_{(M^{\mathcal{C}},\Delta)}$. Let us further assume
that given a morphism $\varphi : \Delta \rightarrow \Delta'$ of
$Sign_1$, the induced functor $M^{\mathcal{C}}(\varphi)$ makes the
following diagram commute.
\begin{equation*}
  \xymatrix{
    M^{\mathcal{C}}(\Delta') \ar[rr]^{M^{\mathcal{C}}(\varphi)} \ar[dr]_{U} 
    && M^{\mathcal{C}} (\Delta) \ar[dl]^{U} \\
    & \mathbf{Set}
   }
\end{equation*}

\noindent
This leads to a functor
$M^{\mathcal{C}} (M^{\mathcal{I}}) : (Sign_1 \times Sign_2)^{op}
\rightarrow \mathbf{Cat}$ such that given a pair
$(\Delta,\Sigma) \in Sign_1 \times Sign_2$,
$M^{\mathcal{C}}(M^{\mathcal{I}}) (\Delta,\Sigma)$ forms a
\emph{discrete} category whose objects are triples $(S, R, m)$ where
$R \in M^{\mathcal{C}}(\Delta)$, $U(R) = S$, and
$m : S \rightarrow M^{\mathcal{I}}(\Sigma)$. Moreover, given a
signature morphism
$\varphi_1 \times \varphi_2 : (\Sigma,\Delta) \rightarrow
(\Sigma',\Delta')$ we have
$M^{\mathcal{C}} (M^{\mathcal{I}}) (\varphi_1 \times \varphi_2) \> (S,
R, m) \triangleq (S, \> M^{\mathcal{C}} (\varphi_1) (R), \>
M^{\mathcal{I}} (\varphi_2) \comp m)$.

\begin{definition}
  An asymmetric combination $\mathcal{C}$ is a tuple
  $(Sign^{\mathcal{C}}, Sen^{\mathcal{C}}, M^{\mathcal{C}},
  \models^{\mathcal{C}})$ such that
  \begin{itemize}
  \item $Sign^{\mathcal{C}}$ is a category of signatures.
  \item $Sen^{\mathcal{C}}$ is a family of functions 
    \begin{flalign*}
      Sen^\mathcal{C}_{Sign} : (Sign \rightarrow \mathbf{Set}) \rightarrow
      (Sign^{\mathcal{C}} \times Sign \rightarrow \mathbf{Set})
    \end{flalign*}
  \noindent
  indexed by the categories $Sign$ in $\mathbf{Cat}$.
  \item $M^{\mathcal{C}}$ is a functor
  $M^{\mathcal{C}} : (Sign^{\mathcal{C}})^{op} \rightarrow
  \mathbf{Cat}$ as assumed above.
  \item Given functors
    $M^{\mathcal{I}} : Sign^{op} \rightarrow \mathbf{Cat}$,
    $Sen^\mathcal{I} : Sign \rightarrow \mathbf{Set}$,
    $\models^{\mathcal{C}}$
    is a family of relation liftings
    $(\models^{\mathcal{C}}_{( 
        \Delta, \Sigma)})_{(\Delta,\Sigma ) \> \in \> Sign^{\mathcal{C}} \times Sign}$
    \begin{center}
      $\models^{\mathcal{C}}_{(\Delta,\Sigma)} : |M^{\mathcal{I}} (\Sigma)| \times Sen^{\mathcal{I}}(\Sigma) 
      \rightarrow
      |M^{\mathcal{C}}(M^{\mathcal{I}}) \> (\Delta,\Sigma)| \times Sen^{\mathcal{C}} (Sen^\mathcal{I}) (\Delta,\Sigma)$
    \end{center}

 \end{itemize}

 \noindent
Given an institution $\mathcal{I}$, a pre-institution $\mathcal{C} \mathcal{I}$,
corresponding to a specific combination, is obtained as follows.

   \begin{itemize}
   \item  $Sign^{\mathcal{C}\mathcal{I}} \triangleq Sign^\mathcal{C} \times Sign^\mathcal{I}$.
   \item
     $Sen^{\mathcal{C}\mathcal{I}} \triangleq Sen^{\mathcal{C}}
     (Sen^\mathcal{I})$. We will assume that the sentences given by
     $Sen^{\mathcal{C}\mathcal{I}}$ are inductively defined (\ie\ are
     generated by a grammar) so that we can define recursive maps on
     them. Intuitively, their atoms include the sentences of the base
     logic.
   \item $Mod^{\mathcal{C}\mathcal{I}} \triangleq M^{\mathcal{C}}(M^{\mathcal{I}})$.
   \item Given a signature
       $(\Delta,\Sigma) \in |Sign^{\mathcal{C}\mathcal{I}}|$, 
       $\models^{\mathcal{C}\mathcal{I}}_{(\Delta,\Sigma)} \triangleq
       {\models^\mathcal{C}_{( \Delta, \Sigma ) } (\models^\mathcal{I}_{\Sigma})}$.
   \end{itemize}
\end{definition}

\subsubsection{Temporalisation.}

We are now ready to recast the three aforementioned combinations of
logics in the institutional setting. We start with temporalisation
since it is the simplest of the three.

\begin{definition} 

  Given an institution $\mathcal{I}$ the temporalisation process
  returns a pre-institution
  $\mathcal{L} \mathcal{I} = (Sign^{\mathcal{L}\mathcal{I}},
  Sen^{\mathcal{L}\mathcal{I}}, Mod^{\mathcal{L}\mathcal{I}},
  \models^{\mathcal{L}\mathcal{I}})$ defined as

    \begin{itemize}
    \item 

      \textsc{Signatures}. $Sign^{\mathcal{L}\mathcal{I}} \triangleq Sign^{\mathcal{L}}
      \times Sign^\mathcal{I}$, where $Sign^\mathcal{L}$ is the one object
      category $1$. Since $Sign^{\mathcal{L}\mathcal{I}} \cong Sign^{\mathcal{I}}$, no
      distinction will be made, unless stated otherwise, between the
      two signature categories.

    \item 

      \textsc{Sentences}. Given a signature $\Sigma \in$
      $|Sign^{\mathcal{L}\mathcal{I}}|$,
      $Sen^{\mathcal{L}\mathcal{I}}(\Sigma)$ is the smallest set generated by
      grammar
      \begin{flalign*}
        \rho \ni \> \psi \> | \> \neg \rho \> | \> \rho
        \wedge \rho \> | \> X \rho \> | \> \rho \> U \rho
      \end{flalign*}        
      \noindent where $\psi \in Sen^\mathcal{I}(\Sigma)$.  For a signature
      morphism $\varphi : \Sigma \rightarrow \Sigma'$,
      $Sen^{\mathcal{L}\mathcal{I}}(\varphi)$ is a function that, provided a
      sentence $\rho \in Sen^{\mathcal{L}\mathcal{I}}(\Sigma)$, replaces the
      base sentences $\psi$ $($\emph{i.e.} elements of
      $Sen^{\mathcal{I}}(\Sigma))$ occurring in $\rho$ by
      $Sen^{\mathcal{I}}(\varphi)(\psi)$; in symbols
       $Sen^{\mathcal{L}\mathcal{I}}(\varphi)(\rho) = \rho [ 
       \psi \in Sen^{\mathcal{I}}(\Sigma) \> / \>
       Sen^{\mathcal{I}}(\varphi)(\psi) \> ] $
       $($recall that sentences are assumed to be inductively defined$)$.

 \item \textsc{Models}. Given the object $\star \in |1|$,
   $M^{\mathcal{L}}(\star)$ is the category whose $($unique$)$ element
   is the pair $(\mathbb{N}, \suc : \mathbb{N}
   \rightarrow \mathbb{N})$ $(\mathbb{N}$ denotes the set of natural
   numbers$)$ and $U \> (\mathbb{N},
   \suc : \mathbb{N} \rightarrow \mathbb{N})$ is
   $\mathbb{N}$. Hence, the elements of category
   $Mod^{\mathcal{L}\mathcal{I}}(\Sigma)$ are triples
     $(\mathbb{N},
     \suc : \mathbb{N} \rightarrow \mathbb{N}, m)$ $($often denoted by letter $M)$
   where $m : \mathbb{N} \rightarrow | Mod^\mathcal{I} (\Sigma) |$.
   We will often denote $m \> (n)$ by $M_n$.
   \item \textsc{Satisfaction}. 
     Given a signature $\Sigma \in |Sign^{\mathcal{L}\mathcal{I}}|$, $M \in
     |Mod^{\mathcal{L}\mathcal{I}}(\Sigma)|$, $\rho \in
     Sen^{\mathcal{L}\mathcal{I}}(\Sigma)$, $M \models \rho$ iff $M \models^0
     \rho$ where

     \smallskip
    \begin{tabular}{l c l} 
        $M \models^j \psi$ & iff & $M_j \models \psi$
        for $\psi \in Sen^\mathcal{I}(\Sigma)$ \\ $M \models^j \rho \wedge
        \rho'$ & iff & $M \models^j \rho$ and $M \models^j \rho'$ \\ $M
        \models^j \neg \rho$ & iff & $M \not \models^j \rho$ \\ $M \models^j X
        \rho$ & iff & $M \models^{j+1} \rho$ \\ $M \models^j \rho \> U \>
        \rho'$ & iff & for some $k \geq j$, $M \models^k \rho'$ and for all $j
        \leq i < k, \> M \models^i \rho$
    \end{tabular}
       
   \end{itemize}

\end{definition}

\noindent
Note that \emph{temporalised} propositional logic coincides with the classic
\emph{linear temporal logic} (\emph{cf.} \cite{finger}).

\begin{theorem}
  \label{sattemp}
  Temporalised $\mathcal{I}$
  (\emph{i.e.} $\mathcal{L}\mathcal{I}$) is an institution.
\end{theorem}

\begin{proof}
  In appendix.
\end{proof}

\noindent
In the sequel we show that the other two asymmetric combinations
enjoy the same property, which is essential for their characterisation
as endofunctors. Of course, this also entails the possibility of
combining a logic an arbitrary number of times, using any of these
three processes.

\subsubsection{Probabilisation.}

In order to handle probabilistic systems (\emph{e.g.}  \emph{Markov
chains}) probabilisation \cite{probabilization} adds a probabilistic
dimension to logics. In institutional terms,

\begin{definition} 
  Consider an arbitrary institution $\mathcal{I}$. Its probabilised
  version $\mathcal{P}\mathcal{I} = (Sign^{\mathcal{P}\mathcal{I}},$
  $Sen^{\mathcal{P}\mathcal{I}}, Mod^{\mathcal{P}\mathcal{I}},
  \models^{\mathcal{P}\mathcal{I}})$ is defined as follows

\begin{itemize}

   \item 
   \textsc{Signatures.} $Sign^{\mathcal{P}\mathcal{I}} \triangleq 
   Sign^{\mathcal{P}} \times
   Sign^\mathcal{I}$, where $Sign^\mathcal{P}$ is the one object category
   $1$. Since $Sign^{\mathcal{P}\mathcal{I}} \cong Sign^{\mathcal{I}}$, 
   no distinction will be made, unless stated otherwise, between the two signature
   categories.
   \item 
   \textsc{Sentences.} For a signature $\Sigma \in |Sign^{\mathcal{P}\mathcal{I}}|$,
   $Sen^{\mathcal{P}\mathcal{I}}(\Sigma)$ is the smallest set generated by grammar
   \vspace{-0.3cm}
   \begin{flalign*}
     \rho \ni \> t < t \> | \> \neg \rho \> | \> \rho \wedge \rho
   \end{flalign*}
   \noindent
   for $t \in \textsc{T}(\Sigma)$ $(\textsc{T} :
   Sign^{\mathcal{P}\mathcal{I}} \rightarrow
   \mathbf{Set})$. $\textsc{T}(\Sigma)$ is generated by grammar
   \begin{center}
   $t \ni \> r \> | \> \int \psi \> | \> t + t \> | \> t \> . \> t$
   \end{center}

   \noindent
   where $r \in \mathbb{R}$ is a real number,
   and $\psi \in Sen^\mathcal{I}(\Sigma)$. Also, we have
   \begin{flalign*}
        &Sen^{\mathcal{P}\mathcal{I}}(\varphi)(\rho) \triangleq \rho [ t \in
        \textsc{T}(\Sigma) \> / \> \textsc{T}(\varphi)(t)\> ], \> \mbox{where}
        \\ & \textsc{T}(\varphi)(t) \triangleq 
        t[ \> \psi \in Sen^\mathcal{I}(\Sigma) \> / \>
        Sen^{\mathcal{I}}(\varphi)(\psi) \> ]
   \end{flalign*}

   \item \textsc{Models.} $Mod^{\mathcal{P}}(\star)$ is the 
     discrete category whose elements are probability spaces
     $(S, p : {2^S} \rightarrow [0,1])$. Functor $U$ returns
      the carrier set.
     Hence, models in $Mod^{\mathcal{P}\mathcal{I}}(\Sigma)$ are triples
     $(S, p, m)$ where $m : S \rightarrow Mod^\mathcal{I}(\Sigma)$.
     For each sentence $\psi \in Sen^\mathcal{I}(\Sigma)$ we set
      $m^{-1} [\psi] \triangleq \{ s \in S : m(s)
      \models \psi \}$.

   \item 
    \textsc{Satisfaction.} Finally, given a signature $\Sigma \in
    |Sign^{\mathcal{P}\mathcal{I}}|$, a model $M \in
    |Mod^{\mathcal{P}\mathcal{I}}(\Sigma)|$, and $\rho \in
    Sen^{\mathcal{P}\mathcal{I}}(\Sigma)$, define
    
    \vspace{-0.2cm}
    \begin{multicols}{2}
      \begin{tabular}{ l c l }

       $M_r$ & = & $r$ \\
       $M_{(\int \psi)}$ & = & $p(m^{-1}[\psi])$ \\
       $M_{(t + t')} $ & =& $M_t + M_{t'}$ \\
       $M_{(t.t')} $ & =& $M_t \> . \> M_{t'}$

     \end{tabular}

     \columnbreak

     \begin{tabular}{ l c l }
       & & \\
       $M \models t < t'$ & iff & $M_t < M_{t'}$
       \\
       $M \models \neg \rho$ & iff & $M \not \models \rho$
       \\
       $M \models \rho \wedge \rho'$ & iff &
       $M  \models \rho$ and $M \models \rho'$

     \end{tabular}

   \end{multicols}
 \end{itemize}

\end{definition} 

\begin{theorem}
  \label{satprob}
  Probabilised $\mathcal{I}$
  $($\emph{i.e.} $\mathcal{P}\mathcal{I})$ is an institution.
\end{theorem}
 
\begin{proof}
  We just need to show that the satisfaction condition holds,
  which follows by a simple case-by-case observation.
\begin{enumerate}[(a)]
   \item The strictly less case is a direct consequence of Lemma
     \ref{terms} in Appendix.
   \item The negation and implication cases follow by induction
     on the structure of sentences.
\end{enumerate}
\end{proof}

\begin{Mexample}
\textsc{Probabilised propositional logic} ($\mathcal{P}PL$).
The probabilisation of propositional logic
is the following logic:
\begin{itemize}

\item \textsc{Signatures}. Signatures are sets of propositional symbols $P$.
\item \textsc{Sentences}. Sentences are generated by grammar 
  $\rho \ni
  \> t < t \> | \> \neg \rho \> | \> \rho \wedge \rho$ where $t$ is a
  term generated by grammar $t \ni \> r \> | \> \int \psi \> | \> t +
  t \> | \> t \> . \> t$ for $r \in \mathbb{R}$ and $\psi$ a
  propositional sentence.

\item \textsc{Models}. 
Models are probability spaces equipped with a function whose domain is
the set of outcomes and the codomain the universe of propositional models.
\end{itemize}
\end{Mexample}

\noindent
Intuitively, $\mathcal{P}\textsc{PL}$ offers a probabilistic
`flavour' to propositions. For instance, one may say that the
probability of $p$ holding is less than probability of $q$ holding,
$\int p < \int q$.  Other examples of probabilised logics are
discussed in \cite{probabilization}.

\subsubsection{Hybridisation.}

Hybridisation \cite{hybridisation} (and its variations \emph{e.g.}
\cite{Gaina17}) provides the foundations for handling different kinds
of \emph{reconfigurable systems} (\emph{i.e.}  computational systems
that change their execution modes throughout their lifetime) in a
systematic manner: in brief, the hybrid machinery relates and
pinpoints the different execution modes while the base logic specifies
the properties that are supposed to hold in each particular mode.

Since hybridisation was originally defined in institutional terms we
will just recall here its definition but without nominal
quantification, which yields an asymmetric fragment of the process.
Such a fragment is adopted in \cite{hybridisation} to define
parametrised translations from hybridised institutions into
first-order logic --- the authors of \cite{Dia16} extended this work
to accommodate nominal quantification as well. The same fragment is
the one adopted in \cite{Mad15} to provide a general characterisation
of equivalence and refinement for hybridised logics.

\begin{definition}
Given an institution $\mathcal{I}$,
${\mathcal{H}\mathcal{I}} = 
(Sign^{\mathcal{H}\mathcal{I}},Sen^{\mathcal{H}\mathcal{I}},Mod^{\mathcal{H}\mathcal{I}}, 
\models^{\mathcal{H}\mathcal{I}})$
is defined as

\begin{itemize}
\item \textsc{Signatures.}
  $Sign^{\mathcal{H}\mathcal{I}} \triangleq Sign^\mathcal{H} \times
  Sign^\mathcal{I}$, where $Sign^\mathcal{H}$ is the category
  $\mathbf{Set} \times \mathbf{Set}$ whose objects are pairs of sets
  $(Nom,\Lambda)$. $Nom$ denotes a set of nominal symbols, and
  $\Lambda$ a set of modality symbols.  \vspace{0.1cm}
        \item \textsc{Sentences.} For a signature $(\Delta,\Sigma) \in
          |Sign^{\mathcal{H}\mathcal{I}}|$ $($with $\Delta = (Nom,\Lambda))$,
          $Sen^{\mathcal{H}\mathcal{I}} (\Delta,\Sigma)$ is the smallest set generated by
          grammar

          \begin{center}
           $\rho \ni \> i \> | \> \psi \> | \> \neg \rho \> | \> \rho
           \wedge \rho \> | \> @_i \rho \> | \> \pv{\lambda} \rho$
          \end{center}

          \noindent where $i \in Nom$, $\psi \in Sen^\mathcal{I}(\Sigma)$, $\lambda
          \in \Lambda$. For a signature morphism $\varphi_1 \times \varphi_2 : 
          (\Delta,\Sigma) \rightarrow (\Delta', \Sigma')$,
          nominals, modalities, and base sentences of 
          $\rho \in Sen^{\mathcal{HI}} (\Delta,\Sigma)$ are replaced according
          to $\varphi_1 \times \varphi_2$ by $Sen^{\mathcal{HI}}(\varphi_1 \times \varphi_2)$.

          \vspace{0.1cm}
        \item \textsc{Models.} Given a signature
          $\Delta \in |Sign^{\mathcal{H}}|$, $M^{\mathcal{H}}(\Delta)$
          is the discrete category whose elements are triples
          $(S,(R_i)_{i \in Nom}, (R_\lambda)_{\lambda \in \Lambda})$
          such that $R_i \in S$, and $R_\lambda \subseteq S \times
          S$. Functor $U$ forgets the last two elements, keeping just
          the carrier set.  For any signature morphism
          $(\varphi_1,\varphi_2) : (Nom,\Lambda) \rightarrow (Nom',
          \Lambda')$, we have
          $M^{\mathcal{H}}(\varphi_1,\varphi_2) (S,(R'_i)_{i \in Nom'},
          (R'_\lambda)_{\lambda \in \Lambda'}) \triangleq (S,(R_i)_{i
            \in Nom}, (R_\lambda)_{\lambda \in \Lambda})$, where
      \begin{flalign*}
        R_i = R'_{\varphi_1(i)} \text{ and }
        R_\lambda = R'_{\varphi_2(\lambda)}
      \end{flalign*}

      \vspace{0.1cm}
       \item 
       \textsc{Satisfaction.} Given $(\Delta,\Sigma) \in
       |Sign^{\mathcal{H}\mathcal{I}}|$, a model $M \in
       |Mod^{\mathcal{H}\mathcal{I}} (\Delta,\Sigma)|$ and a sentence $\rho \in
       Sen^{\mathcal{H}\mathcal{I}}(\Sigma)$, the satisfaction relation is defined as
         \begin{center}
           $M \models \rho$ iff $M \models^w \rho$ for all $w \in S$
         \end{center}
         where 
         \smallskip

         \begin{tabular}{l c l}
           $M \models^w i$ & iff & $R_i = w$ for $i \in Nom$ \\
           $M \models^w \psi$ & iff & $m(w) \models \psi$ for $\psi \in
           Sen^\mathcal{I}(\Sigma)$ \\
           $M \models^w \neg \rho$ & iff & $M \not \models^w \rho$ \\
           $M \models^w \rho \wedge \rho'$ & iff & $M \models^w \rho$
           and $M \models^w \rho'$ \\
           $M \models^w @_i \rho$ & iff & $M \models^{R_i} \rho$ \\
           $M \models^w \pv{\lambda} \rho$ & iff & there is some $w' \in W$
           such that $(w,w') \in R_\lambda$ and $M \models^{w'} \rho$ \\
           % $M \models^w \forall x \> . \> \rho$ & iff & for all model
           % $x$--expansions $M'$ of $M$, $M' \models^w \rho$
         \end{tabular}

\end{itemize}
  
\end{definition}

\noindent
The proof that, for any institution $\mathcal{I}$, hybridisation yields another
institution is given in reference \cite{hybridisation}.

\begin{Mexample}

  \textsc{Hybridised propositional logic ($\mathcal{H}PL$)}. 
  Hybridisation of propositional logic returns the following logic.

  \begin{itemize}

  \item \textsc{Signatures}. Signatures are triples
  $(Nom,\Lambda,P)$ where $Nom$ is a set of nominal symbols,
  $\Lambda$ a set of modality symbols, and $P$ a set of propositional
  symbols.
  \item \textsc{Sentences}. Sentences are generated by grammar
  \begin{center}
    $\rho \ni \> i \> | \> \psi \> | \> \neg \rho \> | \> \rho \wedge
    \rho \> | \> @_i \rho \> | \> \pv{\lambda} \rho$ % \> | \> \forall x
%    . \rho'$
\end{center}
   \noindent
   where $i$
   is a nominal, $\lambda$
   is a modality, and $\psi$
   a propositional sentence.  Note that we have two levels of Boolean
   connectives: the ones from propositional logic, and the ones
   introduced by hybridisation. One can, however, `collapse' them
   since they semantically coincide.
 \item 
   \textsc{Models}. Models are triples $(W,R,m)$ such that $W$
   defines the set of worlds, $R$ describes the transitions between
   worlds and names states. Moreover each world $w \in W$ points to a
   propositional model $m(w)$.

   \end{itemize}

\end{Mexample}

\section{Asymmetric combinations of logics as functors}\label{sc:fun}

\subsection{Lifting comorphisms}

In the previous section three combinations of logics were revisited under the light of the theory
of institutions. We intend now to discuss them as translations between
logics. We will do this at the level of the abstract definition of
a combination of logics given above, leading thus to more powerful
results, applicable not only to the three combinations discussed, but
also to any other fitting the characterisation.

Formally, given a comorphism
$(\Phi,\alpha,\beta) : \mathcal{I} \rightarrow \mathcal{I}'$ a
combination process maps $(\Phi,\alpha,\beta)$ into
$\mathcal{C} (\Phi,\alpha,\beta) : \mathcal{CI} \rightarrow
\mathcal{CI}'$.  The strategy for such a lifting is simple:
when transforming signatures, sentences or models, we keep the top
level structure and change the bottom level according to the base
comorphism. Thus,

\begin{definition}\label{df:morlift}
  A comorphism $(\Phi,\alpha,\beta) : \mathcal{I}
  \rightarrow \mathcal{I}'$ is lifted to a mapping
  $(\mathcal{C} \Phi, \mathcal{C} \alpha, \mathcal{C}\beta) : \mathcal{CI}
  \rightarrow \mathcal{CI}'$ as follows:
\begin{itemize}

  \item 

    \textsc{Signatures}.  $\mathcal{C} \Phi : Sign^{\mathcal{C}\mathcal{I}}
    \rightarrow Sign^{\mathcal{C}\mathcal{I}'}$,
    \begin{flalign*}
      \mathcal{C} \Phi \triangleq
      1_{Sign^{\mathcal{C}}} \times \Phi.
    \end{flalign*}
      
   \item

    \textsc{Sentences}. 
    $\mathcal{C} \alpha : Sen^{\mathcal{C} \mathcal{I}}
      \rightarrow Sen^{\mathcal{C} \mathcal{I}'} \comp \mathcal{C} \Phi$, 
      \begin{flalign*}
        (\mathcal{C} \alpha)_{(\Delta, \Sigma )}(\rho) \triangleq
        \rho \> [ \> \psi \in Sen^\mathcal{I}(\Sigma) \> / \> \alpha_\Sigma 
        (\psi)\> ],
      \end{flalign*}
      \noindent  
      for any $(\Delta,\Sigma) \in |Sign^{\mathcal{C}\mathcal{I}}|$.
 
    \vspace{0.2cm}
    \item
     
     \textsc{Models}. 
     $\mathcal{C} \beta : Mod^{\mathcal{C} \mathcal{I}'} \comp \mathcal{C}
       \Phi^{op} \rightarrow Mod^{\mathcal{C} \mathcal{I}}$, 
       \begin{flalign*}
         (\mathcal{C}
       \beta)_{( \Delta, \Sigma ) } \triangleq id \times id \times
       (\beta_\Sigma \comp),
       \end{flalign*}
       \noindent
       for any $(\Delta,\Sigma) \in |Sign^{\mathcal{C}\mathcal{I}}|$.
\end{itemize}

\end{definition}

\noindent
Clearly, $\mathcal{C} \Phi$ is a functor and both
$\mathcal{C} \alpha$, and $\mathcal{C} \beta$ are natural
transformations.

% \begin{lemma}
%   \label{idpres}
%   The lifting of morphism given in Definition \ref{df:morlift}
%   preserves identities, \emph{i.e.}  $\mathcal{C}(1_{Sign^I},
%     1_{Sen^I}, 1_{Mod^I}) = (1_{Sign^{\mathcal{C}I}}),
%     1_{Sen^{\mathcal{C}I}}, 1_{Mod^{\mathcal{C}I}})$

% \end{lemma}

%\begin{proof} %Let us go over each component
%     \begin{enumerate}[(a)]
%\end{proof}

\begin{lemma}
  \label{compres}
  The lifting process, as defined above, preserves identities and
  distributes over composition.
\end{lemma}

\begin{proof}
  In appendix.
\end{proof}

\noindent
To conclude that the three combinations are endofunctors one step
still remains: to show that the lifted arrows are comorphisms. This,
however, entails the need to inspect each specific combination on its
own, as they all lift the satisfaction relation in different
ways. Certainly a fully generic definition would be an interesting
result. However, this turned out to be a surprisingly complex issue,
which furthermore is not essential for the message that we want this
paper to convey.

\begin{theorem}
  \label{combcom}
  If $(\Phi,\alpha,\beta)$ is a comorphism then,
  for any of the three combinations $\mathcal{C}$ discussed above,
  $\mathcal{C}(\Phi,\alpha,\beta)$ is a comorphism as well.
\end{theorem}

\begin{proof}
  In appendix.
\end{proof}

\subsection{Property preservation (conservativity and equivalence)}

The characterisation of asymmetric combinations as endofunctors over
the category of institutions $\mathbf{I}$ provides a sound basis for
the study of property preservation by them.
 Such a study is illustrated in this section in which it
is shown that temporalisation, probabilisation, and hybridisation
preserve conservativity and equivalence. We start with the former
case.

In Computing Science a main reason to study under what conditions a
logic may be translated into another is to seek for the existence of
(better) computational proof support.  In the institutional setting,
suitable translations are often defined by comorphisms, which in many
cases should obey the following condition: whenever completeness is
required, \emph{i.e.} whenever one demands the validation of the
specification against all possible scenarios (models), then the
comorphisms involved must be conservative. Formally,

\begin{definition}
  A comorphism $(\Phi,\alpha,\beta)$ is conservative
  whenever, for each signature $\Sigma \in |Sign^\mathcal{I}|$,
  $\beta_\Sigma$ is surjective on objects.

\end{definition}

\noindent
Let us describe in more detail the relevance of conservativity for
validation. Recall the satisfaction condition placed upon comorphisms.
For a signature $\Sigma \in |Sign^\mathcal{I}|$,
$M \in |Mod^{\mathcal{I}'} \comp \Phi^{op} (\Sigma) |$, and
$\rho \in Sen^\mathcal{I}(\Sigma)$ we have
$\beta_{\Sigma}(M) \models^{\mathcal{I}}_{\Sigma} \rho$ iff
$M \models^{\mathcal{I}'}_{\Phi(\Sigma)} \alpha_{\Sigma}(\rho)$.
Graphically, for each $\Sigma \in |Sign^\mathcal{I}|$

      \[ \xymatrix{ Mod^{\mathcal{I}}(\Sigma) \ar@{-}[rr]^{\models^\mathcal{I}_\Sigma}
      & &Sen^{\mathcal{I}}(\Sigma)\ar[d]^{\alpha_\Sigma} \\ Mod^{\mathcal{I}'}
      \comp \Phi^{op}
      (\Sigma)\ar[u]^{\beta_\Sigma}\ar@{-}[rr]_{\models^{\mathcal{I}'}_{\Phi(\Sigma)}}
      &&Sen^{\mathcal{I}'} \comp \Phi (\Sigma)\\ }
      \]

\noindent
Suppose we want to verify that a sentence
$\rho \in Sen^\mathcal{I}(\Sigma)$ is satisfied by all models
$M \in |Mod^\mathcal{I}(\Sigma)|$. For this we resort to the
comorphism by translating the sentence (through $\alpha$) into the
target logic. The satisfaction condition, once verified, ensures that
if the sentence is satisfied by all models there, then all models in the
image of $\beta_\Sigma$ will satisfy the original sentence. Of course,
if $\beta_\Sigma$ is surjective on objects its image will coincide
with $|Mod^\mathcal{I}(\Sigma)|$, thus proving that the original
sentence is satisfied by all models in $|Mod^\mathcal{I}(\Sigma)|$.

\begin{theorem}
  A lifted conservative comorphism is still conservative.
\end{theorem}

\begin{proof}
  Consider a conservative comorphism $(\Phi,\alpha,\beta)
  : \mathcal{I} \rightarrow \mathcal{I}'$. We want to prove that for
  any signature $(\Delta,\Sigma) \in |Sign^{\mathcal{C}\textsc{I}}|$
  $(\mathcal{C} \beta)_{( \Delta,\Sigma )} = id \times id \times
  (\beta_\Sigma \comp)$ is surjective on objects.  Since identities
  are surjective we just need to show that each $f \in |
  Mod^{\mathcal{I}} (\Sigma) |^S$ has a function $g \in
  |Mod^{\mathcal{I}'} \comp\Phi^{op} (\Sigma)|^S$ such that $f =
  \beta_\Sigma \comp g$.  Clearly, the condition for this to hold
  is that $img(f) \subseteq img(\beta_\Sigma)$, but the only way to
  ensure it is to have $img(\beta_\Sigma) = |Mod^{\mathcal{I}}
  (\Sigma)|$. In other words, $\beta_\Sigma$ must be surjective on
  objects, which is given by the assumption.
\end{proof}

\noindent
Next we show that the application of temporalisation, probabilisation,
and hybridisation to two equivalent logics yields again two equivalent
logics. 
First, recall the definition of equivalence of categories.

\begin{definition} 
  Two categories $\mathbf{C,D}$ are equivalent if there are two
  functors
  $F: \mathbf{C} \to \mathbf{D}, G : \mathbf{D} \to \mathbf{C}$ and
  two natural isomorphisms $\epsilon : FG \rightarrow 1_{\mathbf{D}}$,
  $\eta : 1_{\mathbf{C}} \rightarrow GF$. In these circumstances, an
  equivalence of categories, $G$ (resp. $F$) is the inverse up to
  isomorphism of $F$ (resp. $G$)
\end{definition}

% \begin{corollary}
% \label{coreq}
%   Consider a category $\mathbf{E}$ and suppose that
% $\mathbf{C},\mathbf{D}$ are equivalent. Clearly, the categories
% $\mathbf{E}\times \mathbf{C}$, $\mathbf{E}\times \mathbf{D}$ must also
% be equivalent.
% \end{corollary}

\begin{definition}

  A comorphism $(\Phi,\alpha,\beta)$ is an equivalence of institutions
  if the following conditions hold.

  \begin{itemize}
  \item \textsc{Signatures}.
    $\Phi$ forms an equivalence of categories.
  \item  \textsc{Sentences}. 
    $\alpha$ has an inverse  up to semantical equivalence, \ie a
    natural transformation 
    $\alpha^{-1} : Sen^{\mathcal{I}'} \comp \Phi \rightarrow
    Sen^{\mathcal{I}}$ such that
    for any sentence $\rho \in Sen^\mathcal{I}(\Sigma)$,
    \begin{center}
        $(\alpha^{-1} \comp \alpha) (\rho)
        \models \rho$, $ \> \> \> \> \rho \models
        (\alpha^{-1} \comp \alpha) (\rho)$
    \end{center}

    \vspace{0.1cm}
    or more concisely, $(\alpha^{-1} \comp \alpha) (\rho)
    \models \hspace{-0.4em}| \>\rho$.

    Moreover, for any sentence $\rho \in Sen^{\mathcal{I}'} \comp \Phi
    (\Sigma)$, $(\alpha \comp \alpha^{-1}) (\rho) \models
    \hspace{-0.4em}| \>\rho$.

  \item \textsc{Models}.  $\beta$ has an inverse up to isomorphism,
    \ie, a natural transformation $\beta^{-1}$ such that for any
    $\Sigma \in |Sign^\mathcal{I}|$, functor $\beta_\Sigma^{-1}$ is
    the inverse up to isomorphism of $\beta_\Sigma$.
  \end{itemize}
  
\end{definition}

\noindent
More about equivalence of institutions can be found in \emph{e.g.}
document \cite{mossa07}.

\begin{theorem}
  A lifted equivalence of institutions is still an equivalence of
  institutions.
\end{theorem}

\begin{proof}
  Suppose that $(\Phi,\alpha,\beta)$ is an institution equivalence. Then,

\begin{itemize}
    \item \textsc{Signatures}.  Since $\Phi$ is an equivalence of
      categories, $\mathcal{C} \Phi = 1_{Sign^\mathcal{C}} \times \Phi$
      must be as well.

  \medskip
\item \textsc{Sentences}. Let $(\mathcal{C} \alpha)^{-1}$ be the
  natural transformation $\mathcal{C} (\alpha^{-1})$. Then, to show that
  for any $\rho \in Sen^{\mathcal{CI}} (\Delta,\Sigma)$, property
  $\big ( (\mathcal{C} \alpha)^{-1} \comp \mathcal{C}\alpha \big )
  (\rho) \models \hspace{-0.4em}| \>\rho$ holds is, by definition of
  $\mathcal{C} \alpha$, equivalent to showing that
    \begin{center}
      $ \rho [ \psi \in Sen^\mathcal{I}(\Sigma) \> / \> (\alpha^{-1} \comp
      \alpha) (\psi) \> ] \models \hspace{-0.4em}| \>\rho$
    \end{center}

    \noindent
    This boils down to proving that $(\alpha^{-1} \comp
    \alpha) (\psi) \models \hspace{-0.5em}| \> \psi$, for any $\psi
    \in Sign^\mathcal{I}(\Sigma)$ which is given by the assumption.

    The proof that $\big ( \mathcal{C} \alpha \comp
    (\mathcal{C}\alpha)^{-1} \big )
    (\rho) \models \hspace{-0.4em}| \>\rho$ is analogous.

  \item \textsc{Models}. Finally, we need to show that for any
    $(\Delta,\Sigma) \in |Sign^{\mathcal{C}I}|$,
    $(\mathcal{C}\beta)_{(\Delta,\Sigma)}$ has an
    inverse up to isomorphism. For this
    we lift $\beta^{-1}_\Sigma$ (given by the assumption) into
    $(\mathcal{C} \beta)^{-1}_{({\Delta,\Sigma} )} =
    (id \times id \times \beta^{-1}_\Sigma \comp)$. Since $\beta^{-1}_\Sigma$
    is an inverse up to isomorphism of $\beta_\Sigma$ it is clear that
    $(\mathcal{C}\beta)^{-1}_{({\Delta,\Sigma} )}$ is also an
    inverse up to isomorphism of  
    $(\mathcal{C}\beta)_{({\Delta,\Sigma})}$.

\end{itemize}
\end{proof}

\subsection{Natural transformations}

We consider now natural transformations between asymmetric
combinations of logics, which seem to fit nicely into the picture:
while lifted comorphisms map the bottom level and keep the top one, such
natural transformations map the top and keep the bottom. For example,
take a natural transformation $\tau : \mathcal{L} \rightarrow
\mathcal{H}$. It is clear that each institution $\mathcal{I}$, induces
a comorphism $\tau_{\mathcal{I}} : \mathcal{LI}
\rightarrow \mathcal{HI}$.  Furthermore, naturality expresses
the commutativity of the diagram below
\begin{displaymath}
\centerline{
\xymatrix{
\mathcal{LI} \ar[d]_{\tau_{\mathcal{I}}}  
\ar[rr]^{\mathcal{L} (\Phi,\alpha,\beta ) }
    & & 
\mathcal{LI}' \ar[d]^{\tau_{\mathcal{I}'}} 
    \\
\mathcal{HI}  
\ar[rr]_{\mathcal{H} (\Phi,\alpha,\beta) }
    & &
\mathcal{HI}'
}
}
\end{displaymath}

\noindent
for each comorphism $(\Phi,\alpha,\beta)$.  This means that when
translating a logic whose levels are both mapped by a composition of
natural transformations and lifted comorphisms, it does not matter
which one of the top or bottom levels is taken first.

Let us illustrate this construction through the natural transformation
$\tau : \mathcal{L} \rightarrow \mathcal{H}$, which relates
temporalisation to hybridisation.  We will, for now, disregard the
$until$ ($U$) constructor associated with $\mathcal{L}$, in order to
keep the construction simple. First consider a signature
$N \in |Sign^{\mathcal{H}}|$ such that
$N \triangleq (\{ Init \},\{ After, After^\star, Next \})$.  Then for
any signature $(N,\Sigma) \in |Sign^{\mathcal{H}I}|$ define the full
subcategory of $Mod^{\mathcal{HI}} (N,\Sigma)$ (denoted in the sequel
by $M^{\mathcal{NI}} (N,\Sigma)$) whose objects are triples $(S,R,m)$
subjected to the following rules:
\begin{multicols}{2}
  \begin{tabular}{l }
    $S = \mathbb{N}$ \\
    $R_{Init} = 0$ \\
    $(a,b) \in R_{Next}$ iff $b = \suc (a)$
  \end{tabular}

  \columnbreak
  \begin{tabular}{l }
    \\
      $(a,b) \in R_{After}$ iff $a < b$ \\
     $(a,b) \in R_{After^\star}$ iff $a \leq b$.
  \end{tabular}
\end{multicols}

% \noindent
% Hence, models take the form

% \begin{displaymath}
% \centerline{
%     \xymatrix{ a \ar[r]^{} \ar@(ul,dl)
%         \ar@/^1.0pc/[r]^{} 
%         \ar@/_1.0pc/[rr]_{}
%         &  b \ar@(ul,ur)
%         \ar[r]^{}
%         \ar@/^1.0pc/[r]^{} 
%         & c \ar@(ur,dr) & \dots }
%         }
% \end{displaymath}

% \noindent
% where 

% \begin{itemize}
% \item $R_{Next} = \{ \pv{a,b}, \pv{b,c}, \dots \}$, 
% \item $R_{After} =  \{ \pv{a,b}, \pv{b,c}, \pv{a,c}, \dots 
% \}$,
% \item $R_{After^\star} =  R_{After} \cup 1_{\mathbb{N}}$.
% \end{itemize}

\begin{definition}
  Given an institution $\mathcal{I}$, define an arrow
  $\tau_\mathcal{I} = (\tau_\mathcal{I} \Phi, \tau_\mathcal{I} \alpha
  , \tau_\mathcal{I} \beta)$
  where

   \begin{itemize}
     \item
       \textsc{Signatures}. $\tau \Phi :
       Sign^{\mathcal{LI}} \rightarrow
       Sign^{\mathcal{HI}} \> $ is a functor such that
       $\tau \Phi \> (\Sigma) \triangleq (N, \Sigma)$ and, for any
       signature morphism $\varphi: \Sigma \rightarrow \Sigma'$

       \begin{center}
         $\tau \Phi \> (\varphi): (N, \> \Sigma)  \rightarrow
         (N, \> \Sigma'), \> \> \>
         \tau \Phi \> (\varphi) \triangleq id \times \varphi$
       \end{center}
       
     \item \textsc{Sentences}. Given a signature $\Sigma \in
       |Sign^{\mathcal{LI}}|$, $\tau \alpha :
       Sen^{\mathcal{LI}}(\Sigma) \rightarrow Sen^{\mathcal{HI}}
       \comp \tau \Phi (\Sigma) \> $ is a function such that
       $\tau \alpha (\rho)
       \triangleq @_{Init} \sigma (\rho)$ where

       \smallskip
       \begin{multicols}{2}
         \begin{tabular}{l c l}
           $\sigma (\psi)$ & = & $\psi, \mbox{ for }
                                 \psi \in Sen^{\mathcal{I}}(\Sigma) $\\
           $\sigma (\neg \rho)$ & =  & $\neg \sigma (\rho)$ \\
         \end{tabular}

         \columnbreak
         \begin{tabular}{l c l}
           $\sigma (\rho \wedge \rho')$ & =  & $\sigma (\rho) \wedge \sigma (\rho')$ \\
           $\sigma ( X \rho)$ & =  & $[Next] \; \sigma (\rho)$ \\
         \end{tabular}
       \end{multicols}
       
       % \vspace{-0.35cm}
       % \begin{tabular}{l c l}
       % $\sigma (\rho U \rho')$ & = & $\exists x \> . \> \pv{After^\star} (
       % \> x \wedge \sigma (\rho')) \wedge [After^\star](\pv{After}x \Rightarrow
       % \sigma (\rho))$
       % \end{tabular}

       \noindent
       The proof that $\tau \alpha$ is a natural transformation follows
       through routine calculation.

      % \begin{displaymath}
      % \centerline{
      % \xymatrix{
      % \rho   
      % \ar@{|->}[rr]^<(.3){ (\mathcal{L} \alpha)_\Sigma }
      % \ar@{|->}[dd]_{ (\tau_I \alpha)_\Sigma }    & & 
      % \rho [ \psi \in Sen^I(\Sigma) \> / \> \alpha_\Sigma(\psi) \> ] 
      % \ar@{|->}[d]^{(\tau_I \alpha)_{\Sigma'}}    \\
      % & & @_{Init} \sigma \> \big ( \rho [ \psi \in Sen^I(\Sigma) \> / \>
      % \alpha_\Sigma (\psi)] \> \big )
      % \\
      % @_{Init} \sigma (\rho)  
      % \ar@{|->}[rr]_<(.2){ (\mathcal{H} \alpha)_\Sigma }
      %     & &
      %  @_{Init} \sigma (\rho) [\psi \in Sen^I(\Sigma) \> /
      % \> \alpha_\Sigma(\psi) \> ]
      % }
      % }
      % \end{displaymath}

      % \medskip
      % Since,
      % $ @_{Init} \sigma \> \big ( \rho [ \psi \in Sen^I(\Sigma) \> / \>
      % \alpha_\Sigma (\psi)] \> \big ) =  @_{Init} \sigma (\rho) [\psi
      % \in Sen^I(\Sigma) \> / \> \alpha_\Sigma(\psi) \> ] $

      % \noindent
      % we conclude that $\tau_I$ is a natural transformation.

   \medskip
   \item
     Finally, given a signature $\Sigma \in |Sign^{\mathcal{LI}}|$, arrow
     $\tau \beta : Mod^{\mathcal{HI}} \comp (\tau \Phi)^{op}
     \rightarrow Mod^{\mathcal{LI}}$ is a functor such that
     \begin{center}
       $\tau \beta \> (S,R,m)  \triangleq (\mathbb{N}, \suc : \mathbb{N} \rightarrow
       \mathbb{N}, m)$
     \end{center}

     \noindent
     \medskip
     Clearly, $\tau \beta$ is a natural transformation.
   \end{itemize}
\end{definition}

\begin{theorem}
  \label{natLH}
  $\tau : \mathcal{L} \rightarrow \mathcal{H}$ forms a natural
  transformation whenever $Mod^{\mathcal{HI}}$ $($for any institution
  $\mathcal{I})$ is equal to $ Mod^{\mathcal{NI}}$.
\end{theorem}

\begin{proof}
  In appendix.
\end{proof}

\noindent
In order to include the $until$ constructor we need to add 
\emph{nominal quantification} to hybridisation, which would yield
the translation
\medskip

       \begin{tabular}{l c l}
       $\sigma (\rho U \rho')$ & = & $\exists x \> . \> \pv{After^\star} (
       \> x \wedge \sigma (\rho')) \wedge [After^\star](\pv{After}x \Rightarrow
       \sigma (\rho))$
       \end{tabular}

\medskip
\noindent
Actually, the proof that hybridisation with nominal quantification is
also an endofunctor (and the satisfaction condition for $until$
associated with $\tau$ holds) boils down to a routine
calculation. This means that the theorem above can be replicated, taking
care of the $until$ operator, in a straightforward manner.

\section{Conclusions and future work}\label{sc:con}
Asymmetric combination of logics is a promising tool for the (formal)
development of complex, heterogeneous software systems. This justifies
their study at an abstract level, paving the way to general results
on, for example, property preservation along the combination process.
Often such a study has been made on a case-by-case basis \emph{e.g.}
\cite{modal_diaconescu,renato-sbmf14,neves2016}.  This paper, on the
other hand, surveys a more general, functorial perspective using three
different asymmetric combinations of logics as case-studies.  In
particular, it provided their characterisation as endofunctors over
the category of institutions by showing how to lift comorphisms and
proving that the lifted arrows obey the functorial laws.  This made
clear that not only logics, but also their translations can be
combined.

The development of an institutional, abstract notion of asymmetric
combination of logics proposed in the paper, hints at a set of
directions for future research.  For example, we saw at the
abstract level that conservativity (an important property for safely
`borrowing' a theorem prover) and equivalence are preserved by
combination. However, a full study is still to be done in what regards
preservation of (co)limits, \emph{e.g.} to discuss whether \emph{the
  combination of the product of two logics is equivalent to the
  product of their respective
  combinations} % (in symbols, $\mathcal{C}
% (I \times I') \cong \mathcal{C} I \times \mathcal{C} I'$).

Another research direction was set by J. Goguen in his Categorial
Manifest \cite{categoricalmanifesto}: \emph{``if you have found an
  interesting functor, you might be well advised to investigate its
  adjoints''}.  We studied natural transformations between
such functors and showed that they nicely complement the lifting of
comorphisms: while the latter maps the bottom level and
keeps the top one, the former maps the top and keeps the bottom. We gave an
example of a natural transformation between temporalisation and
hybridisation, but others deserve to be studied as well.  For example,
in document \cite{hybridisation} it is shown how, given a
comorphism from an institution $\mathcal{I}$ to $FOL$, a
comorphism from $\mathcal{HI}$ to $FOL$ can be obtained. More
generally, the current paper shows that comorphisms can be
built by lifting the original comorphism and then composing
it with the `flat' natural transformation
$E : \mathcal{C} \rightarrow 1_{\mathcal{I}}$ (whenever it exists).
Diagrammatically,
\begin{displaymath}
\centerline{
\xymatrix{
\mathcal{I} \ar@{~)}[d]_{\mathcal{C}}  
\ar[rr]^{(\Phi,\alpha,\beta) }
    & & 
\mathcal{I}' \ar@{~)}[d]^{\mathcal{C}} 
    \\
\mathcal{C}\mathcal{I}  
\ar[rr]_{\mathcal{C} (\Phi,\alpha,\beta)}
    & &
\mathcal{C}\mathcal{I}' \ar@/_2pc/[u]_{E}
}
}
\end{displaymath}
\noindent
On a more speculative note, the perspective taken in this paper also
suggests to look at `trivial' asymmetric combinations. For example, it
is straightforward to define \emph{identisation}, in which the added
layer has a trivial structure, but also \emph{trivialisation}
$(\mathcal{T})$, which turns a logic into the trivial one
(technically, the initial object in the category $\mathbf{I}$ of
institutions). The latter case implies that there is a (unique)
natural transformation $\mathcal{T} \rightarrow \mathcal{C}$ to any
combination $\mathcal{C}$.  % Actually, one can even go further and show
% that $\mathcal{T}$ is the initial object in the category of
% endofunctors over $\mathbf{I}$.

From a pragmatic point of view, the incorporation of these ideas into
the \textsc{Hets} platform \cite{hets} paves the way for its effective
use in Software Engineering.  \textsc{Hets} is often
described as a ``motherboard'' of logics where different ``expansion
cards'' can be plugged in. These refer to individual logics (with their
particular analysers and proof tools) as well as to logic
translations. To make them compatible, logics are formalised as
institutions and translations as comorphisms. Therefore \textsc{Hets}
provides an interesting setting for the implementation of the theory
developed in this paper. Again, a specific case --- that of
\emph{hybridisation} --- was already  implemented in the
\textsc{Hets} platform \cite{hybridisationatwork}.

\subsection*{Acknowledgments}

This work is financed by the ERDF - European Regional Development Fund
through the Operational Programme for Competitiveness and
Internationalisation - COMPETE 2020 Programme and by National Funds
through the Portuguese funding agency, FCT - Fundação para a Ciência e
a Tecnologia within projects POCI-01-0145-FEDER-016692,
UID/MAT/04106/2013. Further support was provided by Norte Portugal
Regional Operational Programme (NORTE 2020), under the PORTUGAL 2020
Partnership Agreement through the ERFD in the context of project
NORTE-01-0145-FEDER-000037. Renato Neves was also sponsored by FCT
grant SFRH/BD/52234/2013, and Alexandre Madeira by FCT grant
SFRH/BPD/103004/2014.

\bibliographystyle{splncs03}
\bibliography{biblio}

\section*{Appendix (Proofs)}

\begin{lemma} \label{terms} 
  For a signature morphism $\varphi :
  \Sigma \rightarrow \Sigma'$, any model $M \in |
  Mod^{\mathcal{P}\mathcal{I}}(\Sigma') |$, and any term $t \in T(\Sigma)$,
  $({\redu{M}{\varphi}})_t = M_{{T}(\varphi)(t)}$
\end{lemma}

\begin{proof}
By induction on the structure of terms,
 \begin{enumerate}[(a)]
 \item
 \begin{eqnarray*}
  && ({\redu{M}{\varphi}})_r
  \justl{=}{ interpretation of terms }
  r
  %
 %  \justl{=}{ interpretation of terms }
 %  %
 %  M_r
 % %
 \justl{=}{definition of ${T}(\varphi)$}
 M_{T(\varphi)(r)}
 \end{eqnarray*}
 \item
\begin{eqnarray*}
  && {(\redu{M}{\varphi})} _{\big ( \mbox{$\int $} \psi \big )}
  \justl{=}{ interpretation of terms  }
  p \> ( \> (Mod^\mathcal{I}(\varphi) \> \comp \> m)^{-1} [\psi] \>)
 \justl{=}{ definition of $m^{-1} [\psi]$}
 p \> 
 ( \> \{ s \in S : Mod^\mathcal{I}(\varphi) \comp m (s) \models \psi \} \> )
 \justl{=}{ $\mathcal{I}$ is an institution }
 p \> 
 ( \> \{ s \in S : m (s) \models Sen^\mathcal{I}(\varphi)(\psi) \} \> )
 \justl{=}{ definition of $m^{-1}[\psi]$ }
 p \> 
 (\> m^{-1}[ \> Sen^\mathcal{I}(\varphi)(\psi) \> ] \> )
 \justl{=}{ interpretation of terms }
 M_{ \> \mbox{$\int$} Sen^\mathcal{I}(\varphi)(\psi)  }
 \justl{=}{ definition of ${T}(\varphi)$ }
 M_{ {T}(\varphi)(\mbox{$\int$} \psi)}
 \end{eqnarray*}
\end{enumerate}
 All other cases are straightforward.
\end{proof}

\bigskip
\noindent
Proof of Theorem \ref{sattemp}. By induction on the structure of
sentences, namely for any $\psi \in Sen^\mathcal{I}(\Sigma)$
 
    \begin{eqnarray*}
    && ({\redu{M}{\varphi}}) \models^j \psi
    \justl{\Leftrightarrow}{definition of $\models^{\mathcal{L}\mathcal{I}}$}
    ({\redu{M}{\varphi}})_j \models \psi
       \justl{\Leftrightarrow}{ (reduct) definition
       of $Mod^{\mathcal{L}\mathcal{I}}$ }
    \redu{M_j}{\varphi} \models \psi 
    \justl{\Leftrightarrow}{ $\mathcal{I}$ is an institution }
    M_j \models Sen^{\mathcal{I}}(\varphi)(\psi)
    \justl{\Leftrightarrow}{ definition of $Sen^{\mathcal{L}\mathcal{I}}(\varphi)$,  
      definition of $\models^{\mathcal{L}\mathcal{I}}$}
    %
    %M_j \models Sen^{\mathcal{L}\mathcal{I}}(\varphi)(\psi)
    %
    %\justl{\Leftrightarrow}{ definition of $\models^{\mathcal{L}\mathcal{I}}$}
    %
    M \models^j Sen^{\mathcal{L}\mathcal{I}}(\varphi)(\psi)
    \end{eqnarray*}
All other cases are straightforward.

\bigskip
\noindent

% \noindent
% Proof of Theorem \ref{satprob}.

\bigskip

\noindent
Proof of Lemma \ref{compres}.
We start with preservation of identities.
\begin{enumerate}[(a)]
  \item \textsc{Signatures}.
  \begin{eqnarray*}
  &&  \mathcal{C} (1_{Sign^\mathcal{I}})
  \justl{=}{ definition of $\mathcal{C} \Phi$}
  1_{Sign^\mathcal{C}} \times 1_{Sign^\mathcal{I}}
  \justl{=}{ $Sign^\mathcal{C} \times Sign^\mathcal{I} = 
    Sign^{\mathcal{C}\mathcal{I}}$}
  1_{Sign^{\mathcal{C}\mathcal{I}}}
\end{eqnarray*}    

\item \textsc{Sentences}.
\begin{eqnarray*}
  && \mathcal{C} (1_{Sen^{\mathcal{I}}})_{(\Delta,\Sigma)} (\rho)
  \justl{=}{ definition of $\mathcal{C} \alpha $ }
  \rho [ \> \psi \in Sen^\mathcal{I}(\Sigma) \> / 
  \> (1_{Sen^{\mathcal{I}}})_\Sigma (\psi) \> ] 
  \justl{=}{ definition of $1_{Sen^{\mathcal{I}}}$ }
  \rho  
 \end{eqnarray*}

\item \textsc{Models}.

 \begin{eqnarray*}
  && \mathcal{C} (1_{Mod^\mathcal{I}})_{( \Delta,\Sigma)} 
  \justl{=}{ definition of $\mathcal{C} \beta$ }
  id \times id \times \big ( \> (1_{Mod^\mathcal{I}})_\Sigma \comp \big )
  \justl{=}{ $id \comp m = m$ }
  id \times id \times id
\end{eqnarray*}

\end{enumerate}

\noindent
In the case of distribution over composition, we reason
\begin{enumerate}[(a)]

\item \textsc{Signatures}. 
  $\mathcal{C} (\Phi_2 \comp \Phi_1) = 
  \mathcal{C} \Phi_2 \comp \mathcal{C} \Phi_1$
 
 \begin{eqnarray*}
  && \mathcal{C} (\Phi_2 \comp \Phi_1)
  \justl{=}{ definition of $\mathcal{C} \Phi$}
    1_{Sign^\mathcal{C}} \times (\Phi_2 \comp \Phi_1)
 \justl{=}{ identity, and definition of product}
  (1_{Sign^\mathcal{C}} \times \Phi_2) \comp (1_{Sign^\mathcal{C}} \times \Phi_1)
 \justl{=}{ definition of $\mathcal{C} \Phi$ (twice) }
 \mathcal{C} \Phi_2 \comp \mathcal{C} \Phi_1
\end{eqnarray*}    

\item \textsc{Sentences}. $\mathcal{C} \big ( (\alpha_2 \circ 1_{\Phi_1}) \comp \alpha_1 \big ) = (\mathcal{C} \alpha_2 \circ 1_{\mathcal{C} \Phi_1}) \comp \mathcal{C} \alpha_1$

\begin{eqnarray*}
  && \mathcal{C} ((\alpha_2 \circ 1_{\Phi_1}) \comp \alpha_1) \> (\rho)
  \justl{=}{ definition of $\mathcal{C} \alpha$, and composition of 
  natural transformations }
  \rho [ \psi \in Sen^\mathcal{I}(\Sigma) \> / \> 
  (\alpha_2 \circ 1_{\Phi_1}) \comp \alpha_1 \> (\psi) ]
  \justl{=}{ horizontal composition}
  \rho [ \psi \in Sen^\mathcal{I}(\Sigma) \> / \> \alpha_2 
  \comp \alpha_1 \> (\psi) ] 
  \justl{=}{ composition }
  \> \big (\rho [\> \psi \in Sen^\mathcal{I}(\Sigma) \> / \>  
  \alpha_1 \> (\psi) \> ] \> \big ) \> \>
  [ \> \psi \in Sen^{\mathcal{I}'} \comp \Phi_1 (\Sigma) \> / \> 
  \alpha_2 \> (\psi) \> ]
  %
  % \justl{=}{definition of $\mathcal{C} \alpha$ (twice)}
  % %
  % (\mathcal{C} \alpha_2)_{(\Delta,\Phi_1(\Sigma))} \comp (\mathcal{C} \alpha_1)_{(\Delta,\Sigma )} (\rho)
  %
  \justl{=}{horizontal composition}
  \big ( (\mathcal{C} \alpha_2) \circ 1_{\mathcal{C} \Phi_1} \big) 
  \comp (\mathcal{C} \alpha_1) (\rho)
   \justl{=}{composition of natural transformations }
      \big ( (\mathcal{C} \alpha_2 \circ 1_{\mathcal{C} \Phi_1}) \comp \mathcal{C} \alpha_1 \big ) (\rho)
 \end{eqnarray*}

\item \textsc{Models}. $\mathcal{C} \big (\beta_1 \comp (\beta_2 \circ 1_{\Phi^{op}_1}) \big ) = \mathcal{C} \beta_1 \comp (\mathcal{C} \beta_2 \circ 1_{\mathcal{C} \Phi^{op}_1})$

\begin{eqnarray*}
  && \mathcal{C} \big (\beta_1 \comp (\beta_2 \circ 1_{\Phi^{op}_1}) \big )_{ ( \Delta,\Sigma ) }
  \justl{=}{definition of $\mathcal{C} \beta$}
  id \times id \times \big ( (\beta_1 \comp (\beta_2 \circ 1_{\Phi^{op}_1})_\Sigma \comp \big )  
  \justl{=}{identity,  and definition of product}
  \big ( id \times id \times (\beta_1)_\Sigma \comp  \big ) \comp 
  \big ( id \times id \times  (\beta_2 \circ 1_{\Phi^{op}_1})_\Sigma \comp \big )
  \justl{=}{horizontal composition}
  \big ( id \times id \times   (\beta_1)_\Sigma \comp  \big ) \comp 
  \big ( id \times id \times  (\beta_2)_{\Phi^{op}_1(\Sigma)} \comp \big )
  \justl{=}{definition of $\mathcal{C} \beta$ (twice)}
  (\mathcal{C} \beta_1)_{ ( \Delta,\Sigma )} \comp (\mathcal{C} \beta_2)_{( \Delta,\Phi^{op}_1(\Sigma) ) }
  \justl{=}{horizontal composition}
  (\mathcal{C} \beta_1)_{( \Delta,\Sigma )} \comp 
  (\mathcal{C} \beta_2 \circ 1_{\mathcal{C} \Phi^{op}_1})_{( \Delta,\Sigma )}
  \justl{=}{composition of natural transformations }
  \big ( \> \mathcal{C} \beta_1 \comp (\mathcal{C} \beta_2 \circ 1_{\mathcal{C}
    \Phi^{op}_1}) \> \big )_{( \Delta,\Sigma )}
\end{eqnarray*}    
\end{enumerate}

\bigskip

\noindent
Proof of Theorem \ref{combcom}. We start with the case of
temporalisation, which follows by induction on the structure of
sentences.
  \begin{enumerate}[(a)]
  \item $\psi \in Sen^\mathcal{I}(\Sigma)$,
  \begin{eqnarray*}
  && (\mathcal{L} \beta) (M) \models^j \psi
  \justl{\Leftrightarrow}{ definition $\models^{\mathcal{L}\mathcal{I}}$}
   (\mathcal{L} \beta) (M)_j \models \psi
  \justl{\Leftrightarrow}{ definition of $\mathcal{L} \beta$ }
  \beta (M_j) \models \psi
  \justl{\Leftrightarrow}{ $(\Phi,\alpha,\beta)$ is a
    comorphism}
  M_j \models \alpha (\psi)
  \justl{\Leftrightarrow}{ definition of $\mathcal{L} \alpha$}
  M_j \models (\mathcal{L} \alpha) (\psi)
 \justl{\Leftrightarrow}{ definition of $\models^{\mathcal{L}\mathcal{I}}$ }
 M \models^j (\mathcal{L} \alpha) (\psi)
\end{eqnarray*}    

\item $\neg \rho$,
  \begin{eqnarray*}
  && (\mathcal{L} \beta) (M) \models^j \neg \rho
  \justl{\Leftrightarrow}{ definition $\models^{\mathcal{L}\mathcal{I}}$ }
  (\mathcal{L} \beta) (M) \not \models^j  \rho
  \justl{\Leftrightarrow}{ induction hypothesis }
  M \not \models^j (\mathcal{L} \alpha) (\rho)
  \justl{\Leftrightarrow}{ definition of $\models^{\mathcal{L}\mathcal{I}}$ and
    $\mathcal{L} \alpha$ }
  M \models^j (\mathcal{L} \alpha) (\neg \rho)
\end{eqnarray*}    
\end{enumerate}
The remaining cases are analogous.  For the case of probabilisation
we need the result described in the lemma below.

\begin{lemma}
\label{terms2}
Consider a signature $\Sigma \in |Sign^{\mathcal{P}\mathcal{I}}|$, a
term $t \in \textsc{T}(\Sigma)$, and a model $M \in |
Mod^{\mathcal{P}\mathcal{I}'} \comp \mathcal{P} \Phi^{op} (\Sigma)|$.
The following property holds.
  \begin{center}
    $\big ((\mathcal{P} \beta) \> (M) \big )_t = M_{
      (\mathcal{P} \alpha) (t) }$
  \end{center}

\end{lemma}

\begin{proof}
 Follows by induction on the structure of terms.
  \begin{enumerate}[(a)]
  \item
\begin{eqnarray*}
  && { \big ( (\mathcal{P} \beta) (M) \big ) }_r
  \justl{=}{ interpretation of terms }
  %
  % r
  % %
  % \justl{=}{ interpretation  of terms }
  % %
  M_r
  \justl{=}{ definition of $\mathcal{P} \alpha$ }
  {M}_{ (\mathcal{P} \alpha) (r)}
 \end{eqnarray*}
 \item 
\begin{eqnarray*}
  && {\big ( \> (\mathcal{P} \beta) (M) \> \big )}_{ \mbox{$\int $} \psi}
  \justl{=}{ definition of $\mathcal{P} \beta$ interpretation of terms}
  p \> \big ( \> ( \beta \> \comp \> m)^{-1} [\psi] \> \big )
 \justl{=}{ definition of $m^{-1} [\psi]$}
 p \> \big 
 ( \> \{ s \in S : \beta \> \comp \> m (s) \models \psi \} \> \big )
 \justl{=}{ $(\Phi,\alpha,\beta)$ is a comorphism }
 p \> \big ( \> \{ \> s \in S : 
 m (s) \models \alpha (\psi) \> \} \> \big )
 \justl{=}{definition of $m^{-1}[\psi]$ }
 p \> \big ( \> m^{-1}[\alpha (\psi)] \> \big )
 \justl{=}{ definition of $\mathcal{P} \beta$ and interpretation of terms}
 {M}_{ \> \mbox{$\int$} \alpha (\psi) }
 \justl{=}{ definition of $\mathcal{P} \alpha$}
 {M}_{ (\mathcal{P} \alpha) (\mbox{$\int$} \psi)}
 \end{eqnarray*}
\end{enumerate}
The remaining cases are proved in a similar fashion.
\end{proof}

\noindent
The satisfaction condition for $\mathcal{P}(\Phi,\alpha,\beta)$
follows by induction on the structure of sentences. In particular,
the stricly less case is a direct consequence of the previous lemma.
Negation and implication are proved as usual.

The case of hybridisation follows again by induction on the
structure of sentences. Thus,

  \begin{enumerate}[(a)]

  \item  $i \in Nom$,
  \begin{eqnarray*}
  && \mathcal{H} \beta \> (M) \models^w i
  \justl{\Leftrightarrow}{ definition of $\models^{\mathcal{H}}$}
  \big ( \> \mathcal{H} \beta \> (M) \> \big )_i = w
  \justl{\Leftrightarrow}{ definition of $\mathcal{H} \beta$}
  M_i = w
  \justl{\Leftrightarrow}{ definition of $\models^{\mathcal{H}}$,
  and $\mathcal{H}\alpha$}
  %
  % M \models^w i
  % %
  % \justl{\Leftrightarrow}{ definition of  $\mathcal{H} \alpha$}
  %
  M \models^w \mathcal{H} \alpha \> (i)
  \end{eqnarray*}
  
  \item $\psi \in Sen^{\mathcal{I}}(\Sigma)$,
  \begin{eqnarray*}
  && \mathcal{H} \beta \> (M) \models^w  \psi
  \justl{\Leftrightarrow}{ definition of $\models^{\mathcal{H}}$ }
  \beta \comp m (w) \models \psi
  \justl{\Leftrightarrow}{ $(\Phi,\alpha,\beta)$ is a
    comorphism  }
  m(w) \models  \alpha \> (\psi)
  \justl{\Leftrightarrow}{ definition of $\models^{\mathcal{H}}$, and
  $\mathcal{H} \alpha$}
  %
%  M \models^w \mathcal{H} \alpha \> (\psi)
  %
%  \justl{\Leftrightarrow}{ definition of   }
  %
  M \models^w \mathcal{H} \alpha \> (\psi)
  \end{eqnarray*}

\item $@_i \rho$,
  \begin{eqnarray*}
  && \mathcal{H} \beta \> (M) \models^w @_i \rho
  \justl{\Leftrightarrow}{  definition of $\models^{\mathcal{H}}$, and
    $\big ( \> \mathcal{H} \beta \> (M) \> \big)_i = M_ i$}
  \mathcal{H} \beta \> (M) \models^{M_i} \rho
  \justl{\Leftrightarrow}{ induction hypothesis}
  M \models^{M_i} \mathcal{H} \alpha \> (\rho)
  \justl{\Leftrightarrow}{  definition of $\models^{\mathcal{H}}$ }
  M \models^w @_i \> \mathcal{H} \alpha \> (\rho)
  \justl{\Leftrightarrow}{ definition of $\mathcal{H} \alpha$}
  M \models^w \mathcal{H} \alpha \> (@_i \rho)
\end{eqnarray*}

  \item $\pv{\lambda} \rho$,
  \begin{eqnarray*}
  && \mathcal{H} \beta \> (M) \models^w \pv{\lambda} \rho
  \justl{\Leftrightarrow}{ definition of $\models^{\mathcal{H}}$, and
  $R_\lambda$ of $\mathcal{H} \beta \> (M)$ 
  is equal to $R_\lambda$ of $M$}
  \mbox{there is a } w' \mbox{ such that } (w,w') \in R_\lambda \mbox{ and
  } \mathcal{H} \beta \> (M) \models^{w'} \rho
  \justl{\Leftrightarrow}{ induction hypothesis}
   \mbox{there is a } w' \mbox{ such that } (w,w') \in R_\lambda \mbox{ and
  } M \models^{w'} \mathcal{H} \alpha \> (\rho)
  \justl{\Leftrightarrow}{ definition of $\models^{\mathcal{H}}$ }
  M \models^w \pv{\lambda} (\mathcal{H} \alpha  \> (\rho))
  \justl{\Leftrightarrow}{ definition of  $\mathcal{H} \alpha$ }
  M \models^w \mathcal{H} \alpha \>  (\pv{\lambda}  \rho)
\end{eqnarray*}    
\end{enumerate}
The remaining cases are routine  induction proofs.

\bigskip

\noindent
Proof of Theorem \ref{natLH}.
Follows by induction on the structure of sentences,
in particular
 \begin{enumerate}[(a)]
  \item $\psi \in Sen^{\mathcal{I}} (\Sigma)$,
\begin{eqnarray*}
  &&  \tau \beta \> (\mathbb{N}, R, m) \> \models^j \psi
  \justl{\Leftrightarrow}{definition of $\tau \beta$}
  (\mathbb{N}, \suc : \mathbb{N} 
  \rightarrow \mathbb{N}, m) \models^j \psi
  \justl{\Leftrightarrow}{definition of $\models^{\mathcal{L}\mathcal{I}}$}
  m(j) \models \psi
  %
  % \justl{\Leftrightarrow}{ definition of $\sigma$, 
  %   definition of $\models^{\mathcal{H}}$ }
  %
%  m(j) \models \sigma (\psi) 
  %
  \justl{\Leftrightarrow}{ definition of $\models^{\mathcal{LI}}$, 
    definition of $\sigma$ }
  (\mathbb{N},R,m) \models^j \sigma (\psi)
 \end{eqnarray*}
  \item $X \rho$,
  \begin{eqnarray*}
  &&  \tau \beta \> (\mathbb{N},R,m) \models^j X \rho
  \justl{\Leftrightarrow}{definition of $\models^{\mathcal{L}\mathcal{I}}$ }
  \tau \beta \> (\mathbb{N},R,m) \models^{j+1} \rho
  \justl{\Leftrightarrow}{ induction hypothesis }
  (\mathbb{N},R,m) \models^{j+1} \sigma (\rho)
  \justl{\Leftrightarrow}{ $R_{Next}$ defines the successor function }
  (\mathbb{N},R,m) \models^j [Next] \> \sigma (\rho) 
  \justl{\Leftrightarrow}{ definition of $\sigma$ }
  (\mathbb{N},R,m) \models^j \sigma( X \rho )
 \end{eqnarray*}
\end{enumerate}
The remaining cases are proved similarly.

\end{document}